\documentclass[10 pt]{amsart}
\usepackage{amsmath}
\usepackage{amstext}
\usepackage{amssymb}
\usepackage{amsthm}
\usepackage{layout}
\usepackage{amscd}
\usepackage[small,nohug,heads=littlevee]{diagrams}
\theoremstyle{plain}
\newtheorem{theorem}{Theorem}[section]
\newtheorem{corollary}[theorem]{Corollary}
\newtheorem{lemma}[theorem]{Lemma}
\newtheorem{proposition}[theorem]{Proposition}
\newtheorem{conjecture}[theorem]{Conjecture}
\theoremstyle{definition}
\newtheorem{definition}[theorem]{Definition}

\newtheorem{remark}[theorem]{Remark}
\newtheorem*{notation}{Notation}
\DeclareMathOperator{\disc}{disc} 
\DeclareMathOperator{\GL}{GL} \DeclareMathOperator{\Ker}{ker}

 \DeclareMathOperator{\Aut}{Aut}

\DeclareMathOperator{\Hom}{Hom} 
\DeclareMathOperator{\Sym}{Sym} 
 \DeclareMathOperator{\Gal}{Gal}
\DeclareMathOperator{\rank}{rank}
\DeclareMathOperator{\Frob}{Frob}
\DeclareMathOperator{\End}{End}
\DeclareMathOperator{\ad}{ad}\DeclareMathOperator{\spec}{Spec}

\DeclareMathOperator{\Ht}{ht}

\DeclareMathOperator{\et}{\mbox{\rm{\'{e}t}}}\DeclareMathOperator{\mult}{\mbox{\rm{mult}}}

\newcommand{\OO}{\mathcal{O}}
\newcommand{\QQ}{\mathbf{Q}}

\newcommand{\ZZ}{\mathbf{Z}}
\newcommand{\CC}{\mathbf{C}}

\newcommand{\F}{\mathbf{F}}

\newcommand{\NN}{\mathbf{N}}

\newcommand{\p}{\mathfrak{p}}
\newcommand{\PP}{\mathbf{P}}
\newcommand{\M}{\mathcal{M}}
\newcommand{\ur}{\mathrm{ur}}

\newcommand{\bad}{\mathrm{bad}}
\newcommand{\good}{\mathrm{good}}
\newcommand{\ol}[1]{\overline{{#1}}}

\begin{document}

%
%

\title[Local torsion on abelian surfaces]{LOCAL TORSION ON ABELIAN SURFACES WITH REAL MULTIPLICATION BY $\QQ(\sqrt{5})$}

\author{ADAM GAMZON}

\date{}

\maketitle

\begin{center}\footnote{The author acknowledges with thanks that part of the research on which the article is based was supported by a Fulbright Post-doctoral Fellowship, awarded by the Fulbright Commission in Israel, the United States-Israel Educational Foundation.}\end{center}
\begin{abstract}
Fix an integer $d\geq 1$.  In 2008, David and Weston showed that, on average, an elliptic curve over $\QQ$ picks up a nontrivial $p$-torsion point defined over a finite extension $K$ of the $p$-adics of degree at most $d$ for only finitely many primes $p$.  This paper proves an analogous averaging result for principally polarized abelian surfaces $A$ over $\QQ$ with real multiplication by $\QQ(\sqrt{5})$ and a level-$\sqrt{5}$ structure.  Furthermore, we indicate how the result on abelian surfaces with real multiplication relates to the deformation theory of modular Galois representations.

\end{abstract}

\section{Introduction}

Let $A$ be an abelian variety over $\QQ$ with an endomorphism ring that embeds into a totally real field. The goal is to give a result providing heuristics for the following conjecture.

\begin{conjecture}[David, Weston]\label{mainconj}
Fix an integer $d\geq 1$.  Then there are finitely many primes $p$ such that $A(K)[p]\neq 0$ where $K$ is a finite extension of $\QQ_p$ of degree at most $d$.
\end{conjecture}

More specifically, we prove that this conjecture holds on average for abelian surfaces with real multiplication by $\QQ(\sqrt{5})$. Indeed, \cite{Mano} shows that the fine moduli space, $X_\QQ$, for principally polarized abelian surfaces over $\QQ$ with real multiplication by $\QQ(\sqrt{5})$ is a double cover of $\PP^2_\QQ$ (ramified over a rational curve of degree 10).  This essentially means that pairs of such abelian surfaces are parameterized by points in $\PP^2(\QQ)$.  It is clear from construction (see Section 5) that one of the abelian surfaces in a fiber of this map has a $p$-torsion point if and only if the other one does.  Let $[a:b:c]$ be homogeneous coordinates on $\PP^2$, let $A_{[a:b:c]}$ be any abelian surface in the fiber over $[a:b:c]$, and let $$\pi^d_{[a:b:c]}(x) = \#\{p\leq x: \mbox{$A_{[a:b:c]}(K)[p] \neq 0$ and $[K:\QQ_p] \leq d$}\}.$$  Furthermore, we can assume that, after scaling, $a,b,c \in \ZZ$ and $\gcd(a,b,c)=1$. Define a height function $H$ on $\PP^2(\QQ)$ by $H([a:b:c]) = \max\{|a|,|b|,|c|\}$. Then we have the following theorem.

\begin{theorem}\label{mainthm}
Let $S_B = \{[a:b:c] \in \PP^2(\QQ): H([a:b:c]) \leq B\}$.  If $B \geq
x^{4/3 + \varepsilon}$ for some $\varepsilon > 0$ then
$$\frac{1}{\#S_B} \sum_{[a:b:c] \in S_B} \pi^d_{[a:b:c]}(x)\ll_d 1
\mbox{ \ as \ } x\rightarrow \infty.$$
\end{theorem}

\begin{remark}
We use the notation $\ll_d$ in Theorem \ref{mainthm} to indicate that although $d$ is fixed in the theorem and througout the paper, the constant on the right hand side changes for different choices of $d$.
\end{remark}

The motivation for Conjecture \ref{mainconj} and Theorem \ref{mainthm} stems from results on elliptic curves and from a conjecture of Barry Mazur regarding the deformation theory of modular Galois representations.  In \cite{David}, David and Weston prove an analogous statement to Theorem \ref{mainthm} in the case of elliptic curves. Let $S_{A,B}$ be the set of all elliptic curves  of the form $y^2 = x^3 + ax +b$ where $|a| \leq A, |b|\leq B$ and let $\pi_E^d(x)$ be the number of primes $p \leq x$ such that $E$ has a nontrivial $p$-torsion point over a finite extension of $\QQ_p$ of degree at most $d$.

\begin{theorem}[David, Weston]\label{ellcurvesanalog}
Fix $d \geq 1$.  Assume $A,B \geq x^{7/4 + \varepsilon}$ for some $\varepsilon > 0$.  Then $$\frac{1}{\#S_{A,B}} \sum_{E \in S_{A,B}} \pi^d_{E}(x)\ll_d 1 \mbox{ \ as \ } x\rightarrow \infty.$$
\end{theorem}

Theorems \ref{mainthm} and \ref{ellcurvesanalog} relate to the deformation theory of Galois representations via a conjecture of Mazur and subsequent work on this conjecture by Weston.  Pick any newform $f = \sum a_nq^n$ of level $N$ and weight $k\geq 2$. Set $K_f = \QQ(\{a_n\})$.  It is well known that this is a number field.  Let $\p$ be any prime of the ring of integers $\OO_{K_f}$ lying over a prime $p \in \ZZ$.  Deligne constructed a (semi-simple) mod $\p$ representation $$\ol{\rho}_{f,\p}: \Gal(\QQ_{S\cup \{p\}}/\QQ) \rightarrow \GL_2(k_{f,\p})$$ where $S$ is a finite set of primes dividing $N$, $\QQ_{S\cup\{p\}}$ is the maximal algebraic extension of $\QQ$ unramified outside of $S\cup \{p\}$, and $k_{f,\p}$ is the residue field $\OO_{K_f}/\p$. 

In \cite{Mazur2}, Mazur conjectured that the deformation theory of the mod $\p$ Galois representation $\ol{\rho}_{f,\p}$ attached to a modular form $f$ is unobstructed for all but finitely many $\p$ when $f$ has weight 2.  Furthermore, it is known that the analogous statement when $f$ has higher weight is true (see \cite{Weston2}).  Let $A_f$ be the abelian variety associated to $f$.  This is an abelian variety  defined over $\QQ$ which can be realized as a quotient of the Jacobian of the modular curve $X_1(N)$.  Moreover, note that $A_f$ has dimension $[K_f:\QQ]$ and admits an action of $\OO_{K_f}$ (i.e., the ring of integers $\OO_{K_f} \hookrightarrow \End_\QQ(A_f)$).  Weston has shown \cite{Weston1} that if $K_f = \QQ$ (that is, $A_f$ is an elliptic curve) then Mazur's conjecture holds when $f$ has weight $2$ if and only if there are only finitely many primes $p$ such that $A_f(L)[p] \neq 0$ where $L$ is a quadratic extension of $\QQ_p$.  Furthermore, when $[K_f:\QQ] > 1$ (that is, when $A_f$ is a higher dimensional abelian variety) and $K_f$ is a totally real number field, we discuss in Section 2 why the same argument that Weston used for connecting unobstructedness to local torsion still holds mutatis mutandis.  Thus, assuming $K_f$ is totally real, Conjecture \ref{mainconj} implies that the deformation theory of $\ol{\rho}_{f,\p}$ is unobstructed for almost all $\p$.

Seen in this light, David and Weston's result on elliptic curves corresponds to the case $K_f = \QQ$ in Mazur's conjecture whereas our result for abelian surfaces with real multiplication by $\QQ(\sqrt{5})$ corresponds to the $K_f = \QQ(\sqrt{5})$ case of Mazur's conjecture.  Thus, the combination of the two results indicates that Mazur's conjecture should hold when $K_f$ is either $\QQ$ or $\QQ(\sqrt{5})$.

The proof of Theorem \ref{mainthm} primarily rests on three algebraic results: Lemmas \ref{W2condition}, \ref{liftsofabsurrm}, \ref{Raynaudram}.  Lemma \ref{W2condition} uses the reduction-exact sequence of an abelian surface defined over an unramified extension $K$ of $\QQ_p$ to give a mod $p^2$ condition for detecting when $A$ has a nontrivial $p$-torsion point defined over $K$.  Lemma \ref{liftsofabsurrm} employs restricted Honda systems over $W/p^2$ (where $W$ is the ring of integers of $K$) to count the number of isomorphism classes of lifts of abelian surfaces over $\F_p$ to $\ZZ/p^2$ that satisfy the condition from our first lemma.  Finally, Lemma \ref{Raynaudram} addresses the issue of the assumption in the first lemma that $K$ is an unramified extension of $\QQ_p$.  Specifically, it shows that if $A$ has a nontrivial $p$-torsion point over a ramified extension $K$ of $\QQ_p$ and $p - 1> d$ then this $p$-torsion point is actually defined over the maximal unramified subextension of $K$.  We then combine these algebraic results to reduce the main analytic calculation to a series of straightforward estimates.  Section 3 lays the groundwork for the lemmas of Section 4.  The proof of Theorem \ref{mainthm} is given in Section 5.

Many thanks are owed to Tom Weston for suggesting this problem and for numerous helpful comments.  I would also like to thank Jenia Tevelev and Siman Wong for several useful discussions.  Thank you to Brian Conrad for his instructive suggestions, especially regarding the use of Raynaud's theorem in Lemma \ref{Raynaudram}.  Finally, we thank the referee for a close and careful reading.

\begin{notation}
For a number field $K$, let $G_K := \Gal(\ol{K}/K)$.  We fix embeddings $\ol{\QQ} \hookrightarrow \ol{\QQ}_p$ and, consequently, fix subgroups $G_p := \Gal(\ol{\QQ_p}/\QQ_p) \hookrightarrow G_\QQ$.  For a Galois module $M$, denote its Tate twist by $M(1)$.
\end{notation}



\section{Galois Representations}

\subsection{$\lambda$-adic Galois representations} 

Let $A$ be an abelian variety of dimension $g$ over a number field $K$.  Assume that $A$ has (maximal) real multiplication by a totally real, degree $g$ extension $E$ of $\QQ$.  That is, assume there is a homomorphism $i:\OO_E \hookrightarrow \End_K(A)$ where $\OO_E$ is the ring of integers of $E$.  We call a prime ideal $\lambda$ in $\OO_E$ a prime of $E$ and let $E_\lambda$ and $\OO_{E,\lambda}$ denote the $\lambda$-adic completions of $E$ and $\OO_E$ respectively.

\begin{definition}
A $\lambda$-adic representation of $G_K$ on a finite dimensional $E_\lambda$ vector space $V$ is a continuous homomorphism $$\rho: G_K \rightarrow \Aut(V).$$
\end{definition}


The representations of interest to us are the ones attached the $\ell$-adic Tate module of $A$.

\begin{definition}
The \emph{$\ell$-adic Tate module} of $A$ is $$T_\ell(A) := \varprojlim A[\ell^n].$$
\end{definition}

Let $V_\ell := T_\ell(A) \otimes \QQ$.  It is well known that $T_\ell(A)$ is a free $\ZZ_\ell$-module of rank $2g$ and, hence, $V_\ell$ is a $2g$-dimensional $\QQ_\ell$ vector space.  Moreover, $G_K$ acts continuously on $A[\ell^n]$ for all $n$ and this action commutes with multiplication by $\ell$, so we get an $\ell$-adic representation $$\rho_\ell: G_K \rightarrow \Aut(T_\ell(A)) \subset \Aut(V_\ell).$$  

In our case, the inclusion $E \subset \End_K(A) \otimes \QQ$, gives even more structure to $T_\ell(A)$ and, hence, $V_\ell$.  Indeed, $\End_K(A) \otimes \QQ$ acts on $V_\ell$, so we can view $V_\ell$ as an $E_\ell := E \otimes \QQ_\ell$ module.  Furthermore, by definition, the endomorphisms in $\End_K(A)$ are defined over $K$, so the action of $G_K$ on $V_\ell$ is $E_\ell$-linear.  The decomposition $E_\ell = \prod_{\lambda|\ell} E_\lambda$ gives a decomposition $$V_\ell = \oplus_{\lambda|\ell} V_\lambda \mbox{\ \ \ where\ \ \ } V_\lambda := V_\ell \otimes_{E_\ell} E_\lambda.$$ Thus the $\ell$-adic representation $\rho_\ell$ can be decomposed as the sum of $\lambda$-adic representations $\rho_\lambda: G_K \rightarrow \Aut_{E_\lambda}V_\lambda$.

\begin{proposition}\label{vlstructure}
As an $E_\ell$-module, $V_\ell$ is free of rank $2$.  Moreover, the $E_\lambda$-dimension of $V_\lambda$ is $2$ for all $\lambda$.
\end{proposition}

\begin{proof}
See Theorem 2.1.1 of \cite{Ribet}.
\end{proof}

\subsection{Unobstructedness and local invariants}  

Let $A = A_f$, $F = K_f$ and $\ol{\rho} = \ol{\rho}_{f,\lambda}$ where $A_f$, $K_f$ and $\ol{\rho}_{f,\lambda}$ are defined as in Section 1.  Assume that $f$ has weight 2.  Note that in this case, $\ol{\rho}$ is the mod $\ell$ reduction of the $\lambda$-adic representation of $A$ as described in Section 2.1 (see \cite[Section 9.5]{Diamond2}).  Let $S$ be the set of all primes dividing the level $N$ of $f$.  We say that the deformation theory of $\ol{\rho}$ is \emph{unobstructed} if $H^2(\Gal(\QQ_{S\cup\{\ell\}}/\QQ), \ad\ol{\rho}) = 0$, where $\ad\ol{\rho}$ denotes the adjoint representation of $\ol{\rho}$.  A combination of Poitou-Tate duality and results on Selmer groups reduces the problem of showing that $H^2(\Gal(\QQ_{S\cup\{\ell\}}/\QQ), \ad\ol{\rho}) = 0$ to the statement that $$H^0(G_p, \ol{\varepsilon}\otimes\ad\ol{\rho}) = 0 \mbox{ for all } p \in S\cup\{\ell\}$$ (see \cite[Section 2]{Weston2}).  Here $\ol{\varepsilon}$ denotes the mod $\ell$ reduction of $\varepsilon$, the $\ell$-adic cyclotomic character.

For simplicity of presentation, assume that $\ell$ is unramified in $E$.  Fix a principal polarization $\omega$ of $A$.  The Weil pairing $e^\omega_\lambda: T_\lambda(A) \times T_\lambda(A) \rightarrow \OO_{E,\lambda}(1)$, where $T_\lambda(A) := \varprojlim A[\lambda^n]$, is an alternating, Galois equivariant, perfect pairing (see \cite{Milne}).  (Here, as in the previous subsection, we use the decomposition $\OO_\ell := \OO_E \otimes_\ZZ \ZZ_\ell \cong \prod_{\lambda|\ell} \OO_{E_\lambda}$ to get a decomposition $$A[\ell^n] = \oplus_{\lambda|\ell} A[\lambda^n]$$ where $A[\lambda^n] := A[\ell^n] \otimes \OO_{E_\lambda}$.)  Using $e_\lambda^\omega$, we get an isomorphism of Galois modules $$\End(T_\lambda(A))(1) \cong T_\lambda(A) \otimes T_\lambda(A).$$  Note that $\ad\ol{\rho} \cong \End(A[\lambda])$ as Galois modules, so $$H^0(G_p, \ol{\varepsilon}\otimes\ad\ol{\rho}) \cong H^0(G_p, A[\lambda]\otimes A[\lambda]).$$  Since $e_\lambda^\omega$ is alternating, we further have $$H^0(G_p, A[\lambda]\otimes A[\lambda]) \cong H^0(G_p, \mu_\ell)\oplus H^0(G_p, \Sym^2A[\lambda]).$$

It is straightforward to see that for $p > 2$, $H^0(G_p,\mu_\ell) \neq 0$ if and only if $\lambda | p-1$.  Therefore, assuming that $\lambda$ does not divide $p-1$, we may focus on $H^0(G_p,\Sym^2A[\lambda])$.  From here, however, we may use the exact proof as in \cite[Lemma 10.15]{Weston1} to see that $H^0(G_p, \Sym^2A[\lambda]) \neq 0$ if and only if $A(L)$ has non-trivial $\lambda$-torsion for some quadratic extension $L$ of $\QQ_p$.  Thus using \cite[Proposition 3.2, Remark 3.3 and Proposition 5.3]{Weston2} for the cases where $p \neq \ell$, we have the following result.

\begin{proposition}\label{unobstructed}
The deformation theory of $\ol{\rho}_{f,\lambda}$ is unobstructed for all but finitely many $\lambda$ if and only if $A(L)[\lambda] \neq 0$ for some quadratic extension $L$ of $\QQ_\ell$.
\end{proposition}

\begin{remark}
The key ingredient that allows Proposition \ref{unobstructed} to follow so nicely from known results is the assumption that $A$ has real multiplication.  This enables us to focus on the components $A[\lambda]$ instead of the whole torsion subgroup $A[\ell]$.
\end{remark}

\begin{remark}
Of course, any $\lambda$-torsion point in $A(L)$ is an $\ell$-torsion point in $A(L)$, hence the connection between Proposition \ref{unobstructed} and our main result, Theorem \ref{mainthm}, on $\ell$-torsion over finite extensions of $\QQ_\ell$.
\end{remark}

\section{Dieudonn\'{e} Theory}

Let $k$ be a finite field of characteristic $p$, let $W:= W(k)$ be the ring of Witt vectors over $k$ and let $K$ be the field of fractions of $W$.  Set $W_n = W/p^n$.

One of our principal algebraic results regards counting the number of isomorphism classes of lifts of  an abelian variety over $k$ to $W_2$.  To do this one can either make use of crystalline Dieudonn\'{e} theory or of restricted Honda systems over $W_2$ (see \cite{Berbec}).  Both theories set up equivalencies of categories between Barsotti-Tate groups over $W_2$ and ``linear algebraic'' objects.  Since in addition to lifting Barsotti-Tate groups to $W_2$ we will also be interested in lifting information about group schemes over $k$ of $p$-powered order, we will use the theory of restricted Honda systems.  

We begin with a brief review of the classification of Barsotti-Tate groups over $k$ by Dieudonn\'{e}-modules.  One can think of the classical theory of complex Lie groups and Lie algebras, which (often) creates a dictionary between problems on Lie groups and linear algebra data coming from Lie algebras, as the meta-idea motivating the construction.  The main reference for this material is \cite{Fontaine}, but \cite{Conrad} is a good introduction to the theory and \cite{Pottharst} gives an accessible overview of the more general setting.

\subsection{Classical Dieudonn\'{e} theory}

\begin{definition}
A \emph{Barsotti-Tate group} $G$ over a scheme $S$ is a group scheme such that
\begin{itemize}
\item $G = \varinjlim G[p^n]$ where $G[p^n]$ is the kernel of multiplication by $p^n$,

\item multiplication by $p$ is an epimorphism on $G$,

\item $G[p]$ is a finite, locally-free group scheme.
\end{itemize}
\end{definition}

Since $G[p]$ is a finite, locally free group scheme, it follows that the order of $G[p]$ is of the form $p^h$ where $h$ is a locally constant function on $S$ with values in $\NN$ and, moreover, $G[p^n]$ has order $p^{nh}$.  In the case that $S = \spec k$, the function $h$ is constant and we call it the \emph{height} of $G$, denoted $\Ht(G)$.

\begin{remark} 
The primary examples of Barsotti-Tate groups we will be interested in are those associated to abelian schemes over $S = \spec k$.  Namely, $A[p^\infty] := \varinjlim A[p^n]$ for a $g$-dimensional abelian variety $A$ over $k$, $\QQ_p/\ZZ_p := \varinjlim \ZZ/p^n\ZZ$ and $\mu_{p^\infty} := \varinjlim \mu_{p^n}$.  Note that $\Ht(A[p^\infty]) = 2g, \Ht(\QQ_p/\ZZ_p) = 1$, and $\Ht(\mu_{p^\infty}) = 1$.
\end{remark}

\begin{remark}\label{BTgpdecomp}
Given a group scheme $G$ over $k$, there is a canonical splitting $$G \cong G_{\et} \times G_{\mult} \times G_{\mbox{ll}},$$ where $G_{\et}$ is the maximal \'{e}tale quotient of $G$, $G_{\mult}$ is the maximal multiplicative subgroup of $G$ and $G_{\mbox{ll}}$ is a group scheme with no non-trivial \'{e}tale quotient nor non-trivial multiplicative subgroup (see, for example, \cite{Demazure} or \cite{Demazure2}).  This decomposition will play an important role in the algebraic results of Section 4.
\end{remark}


Let $\sigma$ denote the automorphism on $W$ (and $K$) that extends the Frobenius automorphism $x \mapsto x^p$ on $k$.   Define the Dieudonn\'{e} ring to be $$D_k = W[F,V]/(FV - p)$$ where $F$ (for Frobenius) and $V$ (for Verschiebung) satisfy $F\alpha = \alpha^{\sigma}F$ and $V\alpha = \alpha^{\sigma^{-1}}V$ for all $\alpha \in W$.  We call a $D_k$-module a \emph{Dieudonn\'{e}} module.  One can associate to each Barsotti-Tate group $G$ over $\spec k$ the $D_k$-module $\M(G) := \Hom(G, \widehat{CW}_k)$ where $\widehat{CW}_k$ denotes the formal affine commutative $k$-group scheme representing the Witt covector $k$-group functor $CW_k$ and the $F$ and $V$ action come from their action on the functor $CW_k$.  This gives an antiequivalence of categories.

\begin{theorem}\label{Dmodpdiv}
The functor $G \leadsto \M(G)$ is an antiequivalence of categories between Barsotti-Tate groups over $k$ and $D_k$-modules that are free $W$-modules of rank $\Ht(G)$.
\end{theorem}

\begin{proof}
See \cite{Fontaine}.
\end{proof}

The dictionary given by Theorem \ref{Dmodpdiv} between group schemes and linear algebra also successfully translates many group-scheme theoretic concepts into the Dieudonn\'{e} module world.

\begin{theorem}\label{etale-connected}
Let $G$ be a Barsotti-Tate group over $k$.  Then $G$ is \'{e}tale if and only if $F$ is bijective on $\M(G)$ and $G$ is connected if and only if the action of $F$ is topologically nilpotent.  Define $\M(G)^\ast = \Hom_W(\M(G),K/W)$ and let $F$ $($respectively $V$$)$ act on $\M(G)^\ast$ as $V^\ast$ $($respectively $F^\ast$$)$ where $V^\ast$ and $F^\ast$ denote the dual action of $V$ and $F$ on $\M(G)$.  This gives $\M(G)^\ast$ the structure of a $D_k$-module and, moreover, there is a natural isomorphism of $D_k$-modules $$\varphi_G: \M(G^\ast) \rightarrow \M(G)^\ast$$ where $G^\ast$ is the Serre dual of $G$.
\end{theorem}

\begin{proof}
Again, see \cite{Fontaine}.
\end{proof}

Fontaine also developed a theory for finite flat $W$-group schemes $G$.  He classifies these via their closed fibers and some ``extra data'' (which we will soon make more explicit).  Similar to the theory for Barsotti-Tate groups, one can associate to $G_k$ a finite Dieudonn\'{e} module $\M(G_k)$.  By a finite $D_k$-module, we mean a $D_k$-module $M$ with finite $W$-length, denoted $\ell_W(M)$.  Fontaine \cite{Fontaine} then shows that $\M(G_k)$ satisfies analogues to Theorems \ref{Dmodpdiv} and \ref{etale-connected}.

\begin{theorem}\label{ffgsoverk}
Let $G$ be a finite $k$-group scheme of $p$-power order.  Then the functor $G \leadsto \M(G)$ is an antiequivalence of categories between finite $k$-group schemes of $p$-power order and finite $D_k$-modules.  Moreover, the order of $G$ equals the order of $\M(G)$ $($i.e., $p^{\ell_W(\M(G))}$$)$.
\end{theorem}

\begin{theorem}\label{ffgetaleconn}
Let $G$ be a finite $k$-group scheme of $p$-power order.  Then $G$ is \'{e}tale if and only if $F(\M(G)) = \M(G)$ and $G$ is connected if and only if the action of $F$ on $\M(G)$ is nilpotent.  Define $\M(G)^\ast = \Hom_W(\M(G),K/W)$ and let $F$ $($respectively $V$$)$ act on $\M(G)^\ast$ as $V^\ast$ $($respectively $F^\ast$$)$ where $V^\ast$ and $F^\ast$ denote the dual action of $V$ and $F$ on $\M(G)$.  This gives $\M(G)^\ast$ the structure of a $D_k$-module and, moreover, there is a natural isomorphism of $D_k$-modules $$\varphi_G: \M(G^\ast) \rightarrow \M(G)^\ast$$ where $G^\ast$ is the Cartier dual of $G$.
\end{theorem}

Returning to the case where $G$ is a finite flat $W$-group scheme, the extra data needed to classify these comes from a $W$-submodule $L(G)$ of ``logarithms'' of $\M(G_k)$.  More precisely, define the category of \emph{finite Honda systems} over $W$ to be the category whose objects are pairs $(L, M)$ where $M$ is a finite $D_k$-module and $L \subset M$ is a $W$-submodule such that
\begin{enumerate}
\item $V|_L : L \rightarrow M$ is injective,

\item the natural $W$-linear composition $$L/p \rightarrow M/p \rightarrow M/F(M)$$ is an isomorphism of $k$-vector spaces.  
\end{enumerate}
The morphisms $\varphi:(L,M)\rightarrow(L',M')$ are pairs $\varphi=(\varphi_L,\varphi_M)$ where $\varphi_M:M\rightarrow M'$ is a $D_k$-module homomorphism and $\varphi_L:L\rightarrow L'$ is a $W$-module homomorphism such that $\varphi_M|_L = \varphi_L$ as $W$-linear homomorphisms.

\begin{theorem}{\cite[Chapter IV, Proposition 5.1]{Fontaine}}
Let $G$ be a finite flat $W$-group scheme.  Assume $p \neq 2$.  Then $G\leadsto (L(G),\M(G_k))$ gives an antiequivalence of categories between finite flat $W$-group schemes and finite Honda systems over $W$.
\end{theorem}

\subsection{Restricted Honda systems over $W_n$}

More recently, Berbec \cite{Berbec} defined categories of finite Honda systems over $W_n$ and of restricted Honda systems over $W_n$.  He then showed how these categories can be used to classify various subcategories of finite flat group schemes over $W_n$.  We will define the category of restricted Honda systems over $W_n$ and then we will state a result about classifying Barsotti-Tate groups over $W_n$ that we use in Section 4.

\begin{definition}
A \emph{restricted Honda system} over $W_n$ is a pair $(L_n, M)$ where $M$ is a finite $D_k$-module and $(L_n, M/p^{n-1}M)$ is a finite Honda system over $W$.  Morphisms in this category are defined in the obvious manner.
\end{definition}

\begin{remark}
Note that restricted Honda systems over $W_n$ is an abelian category since the category of finite Honda systems over $W$ is abelian (see, for example,\cite{Conrad2} or \cite{Fontaine-Laffaille}).
\end{remark}

To a finite flat $W_n$-group scheme $G$, one can contravariantly associate a restricted Honda system $(L_n(G), \M(G_k))$.  (We refer the reader to \cite{Berbec} for the exact definition of $L_n(G)$.)

\begin{proposition}\label{BTWnclass}
Barsotti-Tate groups $G$ over $W_n$ are classified up to isomorphism by $\M(G_k)$ and $(L_n(G[p^{n-1}]),\M(G[p^{n-1}]_k))$.
\end{proposition}

\begin{proof}
This is a combination of Proposition 3.7, Proposition 3.9, Corollary 3.10 and Remark 3.11 of \cite{Berbec}.
\end{proof}

\section{Algebraic Results}

\subsection{A criterion for local torsion} Assume $p > 2.$  Let $k$ be a finite extension of $\F_p$ of degree $d$.  Let $W$, $K$ and $W_2$ be defined as in Section 2.  Define the \emph{$p$-rank} of a finite abelian group $M$ to be the $\F_p$-dimension of $M\otimes_\ZZ \F_p$ (or, equivalently, the $\F_p$-dimension of the $p$-torsion subgroup $M[p]$).  Denote the $p$-rank of $M$ by $\rank_p M$. Our first result gives a condition for when an abelian variety over $K$ has a nontrivial $p$-torsion point defined over $K$.

\begin{lemma}\label{W2condition}
Let $A$ be a $g$-dimensional abelian variety over $K$ of good reduction.  Then 
$$\rank_pA(W_2) = gd \mbox{ if } A(K)[p] = 0$$ and $$gd+1 \leq \rank_pA(W_2) \leq g(d+1) \mbox{ if } A(K)[p]\neq 0.$$
\end{lemma}

\begin{proof}
Consider the commutative diagram with exact rows:
$$\begin{CD}
0 @>>> \widehat{A}(pW) @>>> A(K) @>>> A(k)@>>> 0\\
@.              @VVV            @VVV           @|   \\
0 @>>> \widehat{A}(pW/p^2W) @>>> A(W_2) @>>> A(k) @>>> 0
\end{CD}$$
where $\widehat{A}$ is the formal group of $A$ over $W$.  Since $p>2$, $\widehat{A}(pW) \cong (p\ZZ_p)^{gd}$ and the morphism $\widehat{A}(pW) \rightarrow \widehat{A}(pW/p^2W)$ can be identified with the natural reduction morphism $\ZZ_p^{gd} \rightarrow (\ZZ/p)^{gd}$.  Therefore, taking $p$-torsion and applying the snake lemma gives the following commutative diagram with exact rows:
$$
\begin{CD}
0 @>>> 0 @>>> A(K)[p] @>>> A(k)[p]@>>> (\ZZ/p)^{gd}\\
@.              @VVV            @VVV           @|              @|\\
0 @>>> (\ZZ/p)^{gd} @>>> A(W_2)[p] @>>> A(k)[p] @>>> (\ZZ/p)^{gd}.
\end{CD}
$$
So $\rank_pA(W_2) = gd + \rank_pA(K)[p]$.  The lemma then follows from the structure of $A(K)[p]$ and the fact that $A(K)[p]$ injects into $A(k)[p]$.
\end{proof}

\begin{remark}
The notion of $p$-rank that is usually used when studying abelian varieties $A$ is equivalent (in our notation) to $\dim_{\F_p}A(\ol{k})[p]$.
\end{remark}

\subsection{Lifts of abelian surfaces to $W_2$ with elevated $p$-rank}



Recall Remark \ref{BTgpdecomp}, which states that for a group scheme $G$ over $k$, there is a canonical splitting $$G \cong G_{\et} \times G_{\mult} \times G_{\mbox{ll}},$$ where $G_{\et}$ is the maximal \'{e}tale quotient of $G$, $G_{\mult}$ is the maximal multiplicative subgroup of $G$ and $G_{\mbox{ll}}$ is a group schem with no non-trivial \'{e}tale quotient nor non-trivial multiplicative subgroup.  This fact about the splitting of group schemes over $k$ will be used in the proofs of the next three results.

\begin{proposition}\label{liftsofabsur}
Let $A$ be an ordinary abelian surface over $k$.  Assume that $A(k)[p] \neq 0$.  Then at most $p^{3d} + p^{2d} - p^d$ of the $p^{4d}$ isomorphism classes of lifts of $A$ to an abelian surface $A'$ over $W_2$ satisfy $$\rank_pA'(W_2) \geq 2d+1.$$
\end{proposition}

\begin{proof}

Serre-Tate lifting \cite{Katz} tells us that lifts of $A$ to $W_2$ are parameterized by lifts of $A[p^\infty]$  to a Barsotti-Tate group over $W_2$.  By Proposition \ref{BTWnclass}, classifying lifts of $A[p^\infty]$ to $W_2$ is equivalent to determining all restricted Honda systems $(L_2,M)$ over $W_2$ where $M = \M(A[p])$.

Since $A$ is ordinary, $$A[p] \cong A[p]_{\mbox{\'{e}t}} \times A[p]_{\mult}$$ where $A[p]_{\mbox{\'{e}t}} = (\ZZ/p\ZZ)^2(\rho)$ and $A[p]_{\mult} = (\ZZ/p\ZZ)^2(\rho)^\vee$.  We use the notation $(\ZZ/p\ZZ)^2(\rho)$ to denote the $k$-group scheme that is constant after a finite base change $k'/k$ and whose $\ol{k}$-points admit a Galois action of $\Gal(\ol{k}/k)$ given by the representation $\rho$.  In particular, since $A(k)[p]\neq 0$, we can assume that $\rho$ has the form $$\begin{pmatrix} 1 & \mu\\0 &\chi\end{pmatrix}.$$  Also note that after base changing to a finite extension $k'/k$, the Cartier dual of $(\ZZ/p\ZZ)^2(\rho)$ becomes $\mu_p^2$.  Thus the Dieudonn\'{e} module $\M(A[p])$ splits as \begin{equation}\label{Msplitting}\M((\ZZ/p\ZZ)^2(\rho))\oplus\M((\ZZ/p\ZZ)^2(\rho)^\vee).\end{equation}


By Theorem \ref{ffgsoverk} and equation (\ref{Msplitting}), we know that $M$ is a $k$-vector space of dimension four.  Note that on the \'{e}tale component of $M$, the operator $F$ acts as $\rho(\Frob_k)$.  Pick a basis $e_1, e_2, e_3,$ and $e_4$ so that $e_1, e_2$ form a basis of the \'{e}tale component of $M$ and $e_3, e_4$ form a basis of the connected component as in (\ref{Msplitting}).  Since $pM = 0$, we have that the conditions on $L_2 \subset M/pM = M$ become
\begin{itemize}
\item $V|_{L_2}: L_2 \hookrightarrow M$,

\item $L_2 = L_2/p \rightarrow M/F(M)$ is a $k$-linear isomorphism.
\end{itemize}
Thus, as $F(M) = ke_1\oplus ke_2$ by Theorem \ref{ffgetaleconn}, \begin{equation}\label{L2}L_2 = k(\alpha_1e_1 + \alpha_2e_2 + e_3) \oplus k(\beta_1e_1 + \beta_2e_2 + e_4)\end{equation} for some $\alpha_i, \beta_i \in k$.  Thus there are a total of $p^{4d}$ isomorphism classes of lifts of $A$ to $W_2$.

Now we determine those lifts that satisfy the elevated $p$-rank condition of Lemma \ref{W2condition}.   Let $G$ be a lift of $A[p]$ to $W_2$ (i.e., $G$ is the $p$-torsion subgroup of an abelian surface over $W_2$).  Note that $G$ is an extension of $(\ZZ/p\ZZ)^2(\rho)$ by $(\ZZ/p\ZZ)^2(\rho)^\vee$.  Since we are only interested in an upper bound on the number of isomorphism classes of $G$ satisfying the elevated $p$-rank condition, it suffices to assume that $\rho = 1$ as this maximizes the number of copies of $\ZZ/p\ZZ$ inside of $(\ZZ/p\ZZ)^2(\rho)$ which consequently maximizes the opportunities for $G$ to have elevated $p$-rank.  Then we can check the condition of Lemma \ref{W2condition} by determining if $G \cong \ZZ/p\ZZ \times G'$ where $G'$ is a lift of $\ZZ/p\ZZ\times\mu_p^2$.  In terms of restricted Honda systems over $W_2$, this means that after picking one of the $p^d + 1$ copies of $\ZZ/p\ZZ$ inside of $A[p] = (\ZZ/p\ZZ)^2\times \mu_p^2$ with Dieudonn\'{e} modules corresponding to the lines $k(e_1 + \gamma e_2)$ for $\gamma \in k$ and $ke_2$ in $M$, we want $$(L_2, M) \cong (0,k) \oplus (L_2',M') = (0\oplus L_2', k\oplus M')$$ where $(L_2',M')$ corresponds to $G'$ and the last equality comes from the definition of products in the category of restricted Honda systems over $W_n$.  For concreteness, choose the line $ke_1$.  Since $M$ and $M'$ only depend on the special fiber of $G$ and $G'$, it follows from (\ref{Msplitting}) that $M' = ke_2\oplus ke_3\oplus ke_4$.  Moreover, $L_2'$ must then be a two-dimensional subspace of the form $k(\alpha e_2 + e_3)\oplus k(\beta e_2 + e_4)$ for some $\alpha, \beta \in k$.  Thus (in the notation of (\ref{L2})) we see that the only valid $L_2$ are those where $\alpha_1 = \beta_1=0$.  That is, excluding the case where $\alpha_1=\beta_1=\alpha_2 = \beta_2 = 0$, this gives $p^{2d}-1$ isomorphism classes of lifts to $W_2$ that satisfy the elevated $p$-rank condition.  Similarly, for each line listed above, an entirely analogous calculation also yields $p^{2d} -1$ isomorphism classes of lifts of $A[p]$ to $W_2$ with elevated $p$-rank (again excluding the case where $G$ splits which is common to them all).  Thus there are $$(p^d +1)(p^{2d} -1) + 1 = p^{3d} + p^{2d} - p^d$$ isomorphism classes of lifts with elevated $p$-rank.
\end{proof}

Let $D>0$ be a square-free integer, let $\OO$ be the ring of integers in $\QQ(\sqrt{D})$ and assume $p$ does not divide $\disc(\QQ(\sqrt{D}))$.  Recall that we say that an abelian surface $A$ over $k$ has real multiplication by $\QQ(\sqrt{D})$ if $\OO \hookrightarrow \End_k(A)$.  

\begin{remark}
In terms of the results of this section, there is nothing special about the choice of $\QQ(\sqrt{D})$.  As indicated by Lemma \ref{W2condition} and Proposition \ref{liftsofabsur}, these techniques are completely general and can be used for abelian surfaces with other types of endomorphism rings.  In fact, should one desire to spend the time working out the linear algebra, it should be relatively easy to obtain the analogous results for higher dimensional abelian varieties.  
\end{remark}

\begin{lemma}\label{liftsofabsurrm}
Let $A$ be an ordinary abelian surface over $k$ with real multiplication by $\QQ(\sqrt{D})$.  Suppose that $A(k)[p] \neq 0$.  If $\OO\otimes W \cong W \times W$ then at most $2p^{d} - 1$ of the $p^{2d}$ isomorphism classes of lifts of $A$ to an abelian surface $A'$ over $W_2$ satisfy $$\rank_pA'(W_2) \geq 2d+1.$$  Otherwise, $\OO\otimes W\cong W[\sqrt{D}]$ and there is a unique isomorphism class lifting $A$ to an abelian surface $A'$ over $W_2$ satisfying $$\rank_pA'(W_2) \geq 2d+1.$$
\end{lemma}

\begin{proof}
The argument in this case follows the same logic as in Proposition \ref{liftsofabsur} except that now we must include a condition on $(L_2,M)$ that takes the real multiplication into consideration.  More precisely, we require $L_2$ and $M$ to also have $\OO\otimes W$-module structures.  Note that $$\OO\otimes W \cong \left\{\begin{array}{ll} W\times W, & \mbox{if $p$ splits in $\OO$ or if $[k:\F_p] \equiv 0 \bmod 2$},\\ W[\sqrt{D}], & \mbox{otherwise.}\end{array}\right.$$  We will treat these two possibilities separately.

Since $A$ is ordinary and $A(k)[p] \neq 0$, we see that $$A[p] = (\ZZ/p\ZZ)^2(\rho) \times (\ZZ/p\ZZ)^2(\rho)^\vee$$ over $k$ as in the previous proposition.  In particular, note that as $(\ZZ/p\ZZ)^2(\rho)^\vee$ is connected while $(\ZZ/p\ZZ)^2(\rho)$ is \'{e}tale, each of these direct factors of $A[p]$ must be stable under the action of $\OO$.  Returning to the Dieudonn\'{e} module side of the picture, this means that the action of $\alpha \in \OO\otimes W$ on $M$ has the form $$\begin{pmatrix}B_\alpha&0\\0&C_\alpha\end{pmatrix}$$ where $B_\alpha$ and $C_\alpha$ are $2\times 2$ matrices with coefficients in $k$ (since as a $W$-module $M$ has length one; that is, it is a $k$-vector space).

If $\OO\otimes W \cong W\times W$ then, as a $W$-algebra, $\OO\otimes W$ is generated by two orthogonal idempotents whose sum is the identity; namely, $(1,0)$ and $(0,1)$.  Note that as $(\OO\otimes W)/p = \OO\otimes k$ modules, $M$ is free of rank two while the Dieudonn\'{e} submodules $\M((\ZZ/p\ZZ)^2(\rho))$ and $\M((\ZZ/p\ZZ)^2(\rho)^\vee)$ are both free of rank one as $\OO\otimes k$ modules.  This implies that, up to a change of basis, the action of $(1,0)$ and $(0,1)$ on $M$ can be represented by the matrices $$f_1 =\begin{pmatrix}1&&&\\&0&&\\&&1&\\&&&0\end{pmatrix} \mbox{ and } f_2 = \begin{pmatrix}0&&&\\&1&&\\&&0&\\&&&1\end{pmatrix}$$ respectively.  Then extend $W$-linearly to get a representation of the $W\times W$-action on $M$.  

We also modify the second condition on $L_2$ so that we now require $L_2 = L_2/p$ to be isomorphic to $M/F(M)$ as $\OO\otimes W$-modules.  Let $e_1, e_2, e_3$ and $e_4$ denote the $k$-basis of $M$ chosen so that the generators $(1,0)$ and  $(0,1)$ of $W\times W$ act as $f_1$ and $f_2$ on $M$.  Then this modified condition on $L_2$ implies that $L_2$ is isomorphic to $M/F(M) = \OO\otimes k(e_3 + e_4)$ as $\OO\otimes k$-modules.  (We still have $F(ke_1\oplus ke_2) = ke_1\oplus ke_2$ and $F(e_3) = F(e_4) = 0$ since the chosen basis maintains the decomposition of $M$ into its corresponding connected and \'{e}tale parts.)  Thus $L_2 \subset M$ is of the form  $\OO\otimes k(\alpha e_1 + \beta e_2 + e_3 + e_4)$ for some $\alpha, \beta \in k$.  That is, there are $p^{2d}$ isomophism classes of lifts of $A$ to $W_2$. 


As in Proposition \ref{liftsofabsur}, the greatest number of isomorphism classes of lifts with elevated $p$-rank occur when $\chi = 1$, so we make this assumption.  Then to determine the lifts with elevated $p$-rank, we need to again detect which $(L_2,M)$ admit a splitting; meaning, when we can write $$(L_2,M) = (0,\OO\otimes k(e_1+\gamma e_2)) \oplus (L_2', M') \mbox{ \ \ or \ \ } (L_2,M) = (0,(\OO\otimes k)e_2) \oplus (L_2', M')$$ where $(L_2', M')$ corresponds to a lift of $\ZZ/p\ZZ \times \mu_p^2$ or $\mu_p^2$ and $\gamma \in k$.  If $\gamma \neq 0$ then $(0,\OO\otimes k(e_1 + \gamma e_2))$ must correspond to a lift of $(\ZZ/p\ZZ)^2$ since $\OO\otimes k(e_1 +\gamma e_2)$ is a two dimensional $k$-vector space.  Hence, the only possible splitting of $(L_2, M)$ is the canonical one where $L_2 = \OO\otimes k(e_3 + e_4)$.  If $\gamma = 0$ then $(L_2, M)$ splits if and only if $L_2 = \OO\otimes k(\beta e_2 + e_3 + e_4)$ for some $\beta \in k$.  Similarly, $(L_2, M)$ admits a direct factor of the form $(0,(\OO\otimes k)e_2)$ exactly when $L_2 = \OO\otimes k(\alpha e_1 +e_3+e_4)$ for some $\alpha \in k$.  Thus there are a total of $2p^d -1$ different isomophism classes of lifts with elevated $p$-rank.


Finally, we treat the case where $\OO\otimes W \cong W[\sqrt{D}]$.  Now generators of $\OO\otimes W$ as a $W$-module are $1$ and $\sqrt{D}$.  The action of 1 on $M$ is represented by the identity matrix while the action of $\sqrt{D}$ on $M$ is represented by a matrix of the form $$\begin{pmatrix}B&0\\0&C\end{pmatrix}$$ where $B$ and $C$ are $2\times2$ matrices with entries in $k$ whose minimal polynomials are $X^2 - D$ (since $\sqrt{D} \not\in k$).  Writing this in rational canonical form shows that, up to a change of basis, $$B=C = \begin{pmatrix} 0&D\\1&0\end{pmatrix}.$$  

From here, the argument follows as in the case that $\OO\otimes W \cong W\times W$ except that now there is exactly one isomorphism class lifting $A$ to $W_2$ with elevated $p$-rank.  Indeed, we always have the canonical lift corresponding to the case where $L_2 = \OO\otimes k(e_3 + e_4)$; that is, the situation where $G = (\ZZ/p\ZZ)^2 \times \mu_p^2$.  To show this is the only possibility, it suffices to assume that $\chi =1$ since this is when we have the most options for decomposing $G$ as $\ZZ/p\ZZ\times G'$.  Thus, in terms of restricted Honda systems, we want to show that if $$(L_2,M) = (0,\OO\otimes k(e_1+\gamma e_2)) \oplus (L_2', M') \mbox{ \ \ or \ \ } (L_2,M) = (0,(\OO\otimes k)e_2) \oplus (L_2', M'),$$ where $(L_2', M')$ corresponds to a lift of $\ZZ/p\ZZ \times \mu_p^2$ or $\mu_p^2$ and $\gamma \in k$, then $(L_2',M')$ must correspond to a lift of $\mu_p^2$.  

Suppose $(L_2, M) = (0,\OO\otimes k(e_1 + \gamma e_2)) \oplus (L_2',M')$ for some $\gamma \in k$.  Then $$\sqrt{D}\cdot(e_1+\gamma e_2) = D\gamma e_1 + e_2,$$ so $e_1 + \gamma e_2$ and $\sqrt{D}\cdot(e_1+\gamma e_2)$ are linearly independent over $k$.  That is, the restricted Honda system $(0,\OO\otimes k(e_1 + \gamma e_2))$ corresponds to a lift of $(\ZZ/p\ZZ)^2$ to $W_2$.  Thus the only possibility for $(L_2',M')$ is $(\OO\otimes k(e_3 + e_4), \OO\otimes k(e_3 + e_4))$; i.e., $(L_2, M)$ must split canonically if it splits at all.  An entirely similar argument shows that this is also the case when we replace $\OO\otimes k(e_1 + \gamma e_2)$ with $(\OO\otimes k)e_2$.
\end{proof}

\begin{remark}
The underlying reason or philosophy for why one would expect $p^{2d}$ isomorphism classes of lifts of $A$ to $W_2$ is that Hilbert modular surfaces are moduli spaces for abelian surfaces with real multiplication.  In other words, the moduli space is two-dimensional.
\end{remark}

Finally, if $p$ splits in $\OO$ then an abelian surface $A$ over $k$ with real multiplication by $\QQ(\sqrt{5})$ can  be nonordinary and nonsupersingular (see \cite{Goren}).  Namely, $$A[p] \cong A[p]_{\mbox{ll}} \times \ZZ/p\ZZ(\chi)^\vee \times \ZZ/p\ZZ(\chi).$$

\begin{lemma} \label{nonordlifts}
Suppose that $p$ splits in $\OO$.  Let $A$ be a nonordinary and nonsupersingular abelian surface over $k$ with real multiplication by $\QQ(\sqrt{D})$ and assume that $A(k)[p] \neq 0$.  Then exactly $p^{d}$ of the $p^{2d}$ isomorphism classes of lifts of $A$ to an abelian surface $A'$ over $W_2$ satisfy $$\rank_pA'(W_2) \geq 2d+1.$$  
\end{lemma}

\begin{proof}
In this case, $\OO\otimes W \cong W \times W$.  Since $A(k)[p] \neq 0$, we know that $\chi =1$.  Moreover, note that there are no nontrivial morphisms between local-local, \'{e}tale-local, and local-\'{e}tale group schemes, so the decomposition \begin{equation}\label{lldecomp}M = \M(A[p]_{\mbox{ll}}) \oplus \M(\mu_p) \oplus \M(\ZZ/p\ZZ)\end{equation} of Dieudonn\'{e} modules still obtains as $\OO\otimes W$-modules.  Applying Cartier duality shows that the action of $\OO\otimes W$ on $\M(\mu_p)$ and $\M(\ZZ/p\ZZ)$ must be the same.  Since $M$ is a free $\OO\otimes k$-module of rank 2 and the $\OO \otimes k$-action on $M$ stabilizes each direct summand of (\ref{lldecomp}), this implies that the action of the $k$-algebra generators $(1,0)$ and $(0,1)$ of $k\times k$ on $M$ must be represented by matrices over $k$ of the form $$f_1 =\begin{pmatrix}1&&&\\&1&&\\ && 0&\\&&&0\end{pmatrix} \mbox{ and } f_2 = \begin{pmatrix}0&&&\\&0&&\\&&1&\\&&&1\end{pmatrix}.$$  In other words, one of $(1,0)$ or $(0,1)$ acts as the identity on $\M(A[p]_{\mbox{ll}})$ and trivially on $\M(\mu_p) \oplus \M(\ZZ/p\ZZ)$ while the other acts trivially on $\M(A[p]_{\mbox{ll}})$ and as the identity on $\M(\mu_p) \oplus \M(\ZZ/p\ZZ)$.

It is shown in \cite{Pries} that $\M(A[p]_{\mbox{ll},\ol{k}})/F\M(A[p]_{\mbox{ll},\ol{k}})$ is a one dimensional $\ol{k}$-vector space.  Since the functor $\M$ commutes with base change (see, for example, \cite[Section 4]{Conrad2}), we can conclude that $\M(A[p]_{\mbox{ll}})/F\M(A[p]_{\mbox{ll}})$ is a one dimensional $k$-vector space.  Let $e_1,e_2, e_3$ and $e_4$ be a $k$-basis of $M$ such that $e_1$ and $e_2$ correspond to a basis of the summand $\M(A[p]_{\mbox{ll}})$ where $Fe_1=0$ and $e_2$ is a lift of a basis element of $\M(A[p]_{\mbox{ll}})/F\M(A[p]_{\mbox{ll}})$, while $e_3$ and $e_4$ correspond to bases of the summands $\M(\mu_p)$ and $\M(\ZZ/p\ZZ)$ respectively.  Then $F(M) = \OO\otimes k(e_2 + e_4)$, so $L_2 \cong M/F(M) = \OO\otimes k(e_1 +e_3)$.  We conclude that $L_2$ may be any $\OO\otimes k$-submodule of $M$ of the form $\OO\otimes k(e_1+\alpha e_2 + e_3 + \beta e_4)$.  Thus there are a total of $p^{2d}$ isomorphism classes of lifts to $W_2$.

To determine the lifts with elevated $p$-rank, we need to again detect which pairs $(L_2, M)$ admit a splitting $(0,ke_4) \oplus (L_2', M')$ where $(L_2', M')$ corresponds to a lift of $A[p]_{\mbox{ll}}\times \mu_p$.  This is possible precisely when $\beta = 0$, so there are a total of $p^d$ isomorphism classes of lifts of $A$ to $W_2$ with elevated $p$-rank.
\end{proof}

The following is the main result needed in Section 5.

\begin{corollary}\label{liftingcor}
Let $A$ be an abelian surface over $\F_p$ with real multiplication by $\QQ(\sqrt{5})$.  Suppose that $A(k)[p]\neq 0$.  Then there are at most $2p - 1$ isomorphism classes of lifts to an abelian surface $A'$ over $\ZZ/p^2\ZZ$ satisfying the condition $$\rank_pA'(W_2) \geq 2d +1.$$
\end{corollary}

\begin{proof}
This is a direct consequence of Lemmas \ref{liftsofabsurrm} and \ref{nonordlifts}, and the fact that the functor $G \leadsto (L_2(G),\M(G_{\F_p}))$ for finite flat $\ZZ/p^2\ZZ$-group schemes $G$ commutes with unramified base change by \cite[Theorem 4.8]{Conrad2}.  To see this, suppose we have a finite extension $k/\F_p$ that yields an extension of rings $\ZZ_p \hookrightarrow W$.  Then on restricted Honda systems over $\ZZ/p^2\ZZ$, the base change operator is given by tensoring with $W$ and defining $F$ and $V$ appropriately (which we omit since we will not need $F$ and $V$).  That is, $W_2 \times_{\ZZ/p^2\ZZ} G$ corresponds to the restricted Honda system over $W_2$ $$(W\otimes_{\ZZ_p} L_2(G), W\otimes_{\ZZ_p} \M(G_{\F_p})).$$  Thus the $p^2$ isomorphism classes of lifts of $A[p]$ to $\ZZ/p^2\ZZ$ are determined by the restricted Honda systems over $W_2$ where in the notation of (\ref{L2}), $\alpha_1, \alpha_2, \beta_1,$ and $\beta_2$ all lie in $\ZZ/p^2\ZZ$.  Similarly, to say that $A'$ over $\ZZ/p^2\ZZ$ has $\rank_pA'(W_2) \geq 2d+1$ is the same as saying that $\rank_p(W_2\times_{\ZZ/p^2\ZZ}A')(W_2) \geq 2d+1$.   Therefore, in terms of restricted Honda systems, this means that the $\alpha_i$ and $\beta_i$ from (\ref{L2}) must be as in the proofs of Lemma \ref{liftsofabsurrm} or Lemma \ref{nonordlifts} and, as we just discussed, they must lie in $\ZZ/p^2\ZZ$.  Hence there are at most $2p -1$ possibilities for the $\alpha_i$ and $\beta_i$ that give lifts of $A$ to $\ZZ/p^2\ZZ$ with elevated $p$-rank after base changing to $W_2$.
\end{proof}

\begin{remark}
Although we will eventually require the abelian surfaces to come endowed with a principal polarization, this does not change the numerics of Lemmas \ref{liftsofabsurrm} and \ref{nonordlifts} and, consequently Corollary \ref{liftingcor}.  This is because every principal polarization of an abelian surface over $k$ with real multiplication lifts to a principal polarization of $A'$ by \cite[Corollary 10.1.8]{vanderGeer2}.
\end{remark}

\subsection{Local torsion over ramified extensions of $\QQ_p$}

Up until this point, we have assumed that $K$ is an unramified extension of $\QQ_p$.  The last result of this section enables us to address the possibility that $K$ may be ramified in the statement of Theorem \ref{mainthm}.  

\begin{lemma}\label{Raynaudram}
Let $A$ be an abelian scheme over a complete dvr $R$ with mixed characteristic $(0,p)$ and fraction field $E$.  Suppose that $A[p](K) \neq 0$ for some finite extension $K/E$ with ramification index $e(K/E) < p-1$, then $A[p](K^{\ur}) \neq 0$ where $K^\ur$ is the maximal unramified subfield of $K$.
\end{lemma}

\begin{remark}
The following proof for Lemma \ref{Raynaudram} was suggested to me by Brian Conrad.
\end{remark}

\begin{proof}
Let $Y = A[p]$.  The idea is to show that $Y$ has a nonzero finite \'{e}tale subgroup $H$ since then the generic fiber $H_E$ becomes constant after some finite unramified base change (that is, all geometric points of $H_E$ are defined over $K^\ur$).  To get $H$, first note we may assume that $K/E$ is a Galois extension since $K$ is only tamely ramified over $E$ ($e(K/E) < p-1$) and, hence, its Galois closure over $E$ has the same ramification degree.  Pick any nonzero point $x$ in $Y(K)$.  Then its $G_E$-orbit is a $\Gal(K/E)$-stable subset $S$ of $Y(K)$.  Taking the $\F_p$-span of $S$ gives a $G_E$-submodule of $Y(\ol{E})$ and, thus, a nonzero (finite \'{e}tale) $E$-subgroup scheme of $Y_E$ via Grothendieck's functor (see \cite[Section 3.6]{Tate2}).  Then the scheme-theoretic closure of this $E$-subgroup is a nontrivial finite flat $R$-subgroup scheme $H$ of $Y$ whose generic fiber $H_E$ becomes constant after a base change to $K$.  To see that $H$ itself is finite \'{e}tale, it suffices to show that $H_{\OO_K}$ is finite \'{e}tale, where $\OO_K$ is the ring of integers of $K$.  This, however, follows from an application of Raynaud's theorem \cite[Theorem 4.5.1]{Tate2}, which implies that $H_{\OO_K}$ is constant and, hence, finite \'{e}tale.
\end{proof}

\begin{corollary}\label{ramifiedcase}
Let $A$ be an abelian surface over $\QQ$ with real multiplication by $\QQ(\sqrt{D})$.  Suppose $A(K)[p]\neq 0$ for some $K/\QQ_p$ such that $[K:\QQ_p] \leq d$.  If $p-1>d$ and $A$ has good reduction at $p$, then $A(K^\ur)[p]\neq 0$.
\end{corollary}

\begin{remark}
Corollary \ref{ramifiedcase} may also be proven directly from analyzing the $p$-torsion Galois representation when restricted to a decomposition group at $p$.  Indeed, using the assumption that $A$ has real multiplication, we can reduce the argument to considering a 2-dimensional representation of the form $$\begin{pmatrix} \varepsilon\chi^{-1}& \ast\\0&\chi \end{pmatrix}$$ where $\varepsilon$ is the cyclotomic character, $\chi$ is an unramified character, and $\ast$ is either trivial or wildly ramified.  Finally, since $d < p-1$, the only way $A(K)[p] \neq 0$ is if $\chi$ factors through $\Gal(K/\QQ_p)$ and $\ast$ is trivial on $\Gal(\ol{\QQ_p}/K)$ and, hence, on $G_p$.
\end{remark}

\section{Analytic Methods}

\subsection{Moduli Space of Abelian Surfaces with RM by $\QQ(\sqrt{5})$}

Since it will suffice for our purposes, for simplicity, assume that $S$ is a $\ZZ[1/5]$-scheme.  Families in our moduli problem will be quadruplets $(A/S, \omega, i, (\alpha_1, \alpha_2))$ where
\begin{itemize}
\item $A/S$ is an abelian scheme of relative dimension 2,

\item $\omega$ is a principal polarization,

\item $i:\OO \hookrightarrow \End(A)$ is a homomorphism such that $\omega^{-1}i^\vee(\alpha)\omega = i(\alpha)$ for all $\alpha \in \OO$,

\item $(\alpha_1, \alpha_2)$ defines a level-$\sqrt{5}$ structure (see Definition \ref{levelstructure}).
\end{itemize}
We say $(A/S, \omega, i, (\alpha_1, \alpha_2))$ and $(A'/S, \omega', i', (\alpha_1', \alpha_2'))$ are equivalent if $A$ with its extra structure is isomorphic over $S$ to $A'$ with its extra structure.

\begin{definition}\label{levelstructure}
Let $\delta$ be a root of $x^2 - 5$ and let $A/S$ be as in our moduli problem.  A \emph{level-$\sqrt{5}$ structure} on $A$ is a pair of sections $\alpha_1,\alpha_2: S \rightarrow A$ such that
\begin{enumerate}
\item for all geometric points $s \in S$, the images $\alpha_1(s)$, $\alpha_2(s)$ form a basis of  the group scheme $A_s[\delta]:=\Ker(\psi_\delta)$ where $A_s = A \times s$ is the fiber of $A$ over $s$ and $\psi_\delta = i(\delta):A\rightarrow A$ is the multiplication by $\delta$ morphism,

\item $\psi_\delta\circ \alpha_j = e$ where $e: S \rightarrow A$ is the identity section.
\end{enumerate}
\end{definition}

It is known that a coarse moduli scheme $M$ over $\spec \ZZ[1/5]$ exists for this moduli problem.  Therefore, to show that $M$ is a fine moduli scheme, it suffices to check that families do not have any nontrivial automorphisms, which is exactly what Manoharmayum does in \cite[Propostion 1.1]{Mano}.

\begin{definition}\label{involution}
Let $\tilde{i}$ denote the composition $$\begin{CD} \OO @>\tilde{}>> \OO @>i>> \End_S(A)\end{CD}$$ where $\tilde{}:\OO \rightarrow \OO$ is the Galois involution on $\OO$.  Then define an involution $\dagger$ on $M$ by $(A/S, \omega, i, (\alpha_1,\alpha_2))^\dagger = (A/S, \omega, \tilde{i}, (\alpha_1,\alpha_2))$.
\end{definition}

We can construct a compactification $\ol{M}$ of $M$ by adding six cusps.  Hirzebruch carried out a detailed study of $\ol{M}$ over $\CC$ in \cite{hirzebruch1} and Manoharmayum showed in \cite{Mano} that Hirzebruch's arguments work over $\QQ$.  For our purposes, the main aspect of their work that we need is the following proposition.


\begin{proposition}\label{modspisomtop2}
The quotient $\ol{M}/\dagger$ is isomorphic over $\QQ$ to $\PP^2$.  Under this isomorphism, $\ol{M}$ is a double cover of $\PP^2$ and the six singular points of $\ol{M}$ $($corresponding to the cusps$)$ give a collection of six points of $\PP^2$ defined over $\QQ$.  
\end{proposition}

Finally, note that since $\ol{M}/\dagger$ and $\PP^2$ are schemes of finite type over $\spec \ZZ[1/5]$, the isomorphism $\ol{M}/\dagger \rightarrow \PP^2$ of Proposition \ref{modspisomtop2} is actually defined over $\ZZ[1/N]$ for some $N\geq 1$.  Therefore, in what follows, we will assume that all primes $p > 5$ and $p$ does not divide $N$.

\subsection{Proof of Theorem \ref{mainthm}} 

Continuing with the notation of Section 4, let $k$ be a finite field of degree $d$, $W = W(k)$ its ring of Witt vectors, $K$ the fraction field of $W$ and $W_2 = W/p^2$.  


\begin{definition}
Let $$\nu_d(p) = \#\{y \in \PP^2(\ZZ/p^2): \rank_p A_{y}(W_2) \geq 2d+1\}$$ where $A_{y}$ is the abelian surface in the fiber over $y$.
\end{definition}

\begin{remark}
A priori, $A_{y}$ may not seem to be well-defined, but recall that $(A,\omega, i, \alpha_1, \alpha_2)^\dagger = (A, \omega, \tilde{i}, \alpha_1,\alpha_2)$, so the underlying abelian surface remains the same in the fiber over $y$.
\end{remark}

\begin{lemma}\label{asymptoticbounds} We have
\begin{itemize}
\item[$(a)$] $\sum_{p\leq x} \nu_d(p) \ll x^4$,

\item[$(b)$] $\sum_{p\leq x} \frac{\nu_d(p)}{p^2} \ll x^2$,

\item[$(c)$] $\sum_{p\leq x} \frac{\nu_d(p)}{p^4} \ll x$,

\item[$(d)$] $\sum_{p\leq x} \frac{\nu_d(p)}{p^6} \ll 1$.
\end{itemize}
\end{lemma}

\begin{proof}
For (a), note that by Lemma \ref{W2condition}, \small $$\nu_d(p) = \sum_{\begin{array}{cc}y\in \PP^2(\F_p)\\ A_{y}(k)[p] \neq0\end{array}}\# (\mbox{lifts $A'$ of $A_{y}$ over $\F_p$ to $\ZZ/p^2$ with $\rank_p A'(W_2) \geq 2d+1$}).$$ \normalsize  In particular, Corollary \ref{liftingcor} yields the bound $\nu_d(p)\leq (2p-1)(p^2 + p+1),$ so $$\sum_{p\leq x} \nu_d(p) \ll x^4.$$  For the remaining bounds in (b), (c) and (d) use the estimates $$\frac{\nu_d(p)}{p^2} \ll p, \frac{\nu_d(p)}{p^4} \ll \frac{1}{p} < 1  \mbox{ \ and \ } \frac{\nu_d(p)}{p^6} \ll \frac{1}{p^3}.$$
\end{proof}

Recall that $\pi_y^d(x) = \#\{p\leq x: \mbox{$A_{y}(K)[p] \neq 0$ and $[K:\QQ_p] \leq d$}\}$.  Let $\pi_{y}^{d,{\good}}(x)$ denote the number of $p\leq x$ such that $A_y$ has a $p$-torsion point over an extension of $\QQ_p$ of degree at most $d$ and such that $A_{y}$ has good reduction at $p$.  Similarly, set $$\pi_{y}^{d,\bad}(x) = \#\{p\in \pi_{y}^d(x): \mbox{$A_y$ has bad reduction at $p$}\}.$$  So we now have $$\pi_{y}^d(x) = \pi_{y}^{d,\good}(x) + \pi_{y}^{d,\bad}(x).$$  Lastly, recall that $$S_B = \{y\in \PP^2(\QQ): H(y) \leq B\}$$ where, without loss of generality, the coordinates of $y$ are relatively prime integers $y_1, y_2$ and $y_3$.  As in the introduction, we define the height of a point in $\PP^2(\QQ)$ to be $H(y) = \max\{|y_1|,|y_2|,|y_3|\}$.   We are now ready to prove Theorem \ref{mainthm}, which we restate below.

\begin{theorem}\label{mainthmrestate}
If $B \geq x^{4/3 + \varepsilon}$ for some $\varepsilon > 0$ then $$\frac{1}{\#S_B} \sum_{[a:b:c] \in S_B} \pi^d_{[a:b:c]}(x)\ll_d 1 \mbox{ as } x\rightarrow \infty.$$
\end{theorem}

\begin{proof}
Write 
\begin{equation*}\label{goodbadeqn} 
\frac{1}{\#S_B} \sum_{y \in S_B} \pi^d_{y}(x) = \frac{1}{\#S_B} \left(\sum_{y \in S_B} \pi_{y}^{d,\good}(x) +\pi_{y}^{d,\bad}(x)\right).
\end{equation*}
Corollary \ref{ramifiedcase} implies that the right-hand-side equals 
$$\frac{1}{\#S_B}\left( \sum_{y \in S_B} \pi_{y}^{d,\ur}(x) + O(1) +   \sum_{y \in S_B} \pi_{y}^{d,\bad}(x)\right),$$
 where $\pi_{y}^{d,\ur}(x)$ denotes the number of $p\leq x$ such that $A_{y}$ has a $p$-torsion point over an unramified extension of $\QQ_p$ of degree at most $d$ and such that $A_{y}$ has good reduction at $p$.  

It is relatively easy to show that the sum $$\frac{1}{\#S_B} \sum_{y \in S_B} \pi_{y}^{d,\bad}(x)$$ is analytically irrelevant.  Indeed, for the primes of bad reduction, note that the six cusps of $\ol{M}$ correspond to six points in $\PP^2$ where there is no moduli interpretation (see \cite{Mano}).  That is, $A_y$ will have bad reduction at $p$ if and only if $y$ reduces to a point in $\PP^2(\F_p)$ corresponding to a cusp of $\ol{M}(\F_p)$.  With this in mind, let $\pi_{p}^{d,\bad}$ be the number of points $y$ in $S_B$ such that $A_{y}$ has a nontrivial $p$-torsion point over a degree $d$ extension of $\QQ_p$ and such that $A_{y}$ has bad reduction at $p$.  Assume that $A_{y}(K)[p]\neq  0$ whenever $A_{y}$ has bad reduction at $p$, and that $(2B/p + O(1))^3$ points in $S_B$ reduce to a given point in $\PP^2(\F_p)$.  Then even with these most naive assumptions, reversing the order of summation gives
\begin{eqnarray*}
\frac{1}{\#S_B} \sum_{y \in S_B} \pi_{y}^{d,\bad}(x)& = &\frac{1}{\#S_B} \sum_{p\leq x} \pi_p^{d,\bad}\\
& \leq &\sum_{p\leq x} \frac{6(2B/p + O(1))^3}{\#S_B}.
\end{eqnarray*} Since Schanuel \cite{Schanuel} showed that $\# S_B = B^3/\zeta(3) + O(B^2)$ where $\zeta(s)$ is the Riemann zeta function, this average is finite as $x\rightarrow \infty$.

For the sum $$\frac{1}{\#S_B}\sum_{y \in S_B} \pi_{y}^{d,\ur}(x),$$ let $K_{d_0}$ be the unramified extension of $\QQ_p$ of degree $d_0$ and let $\widetilde{\pi}_{y}^{d_0}(x)$ denote the number of primes $p \leq x$ such that $A_{y}(K_{d_0})[p] \neq 0$ and $A_y$ has good reduction at $p$.  Let $\pi_{p}^{d_0,\good}$ be the number of points  $y \in S_B$ such that $A_y(K_{d_0})[p] \neq 0$ and $A_y$ has good reduction at $p.$  Finally, note the naive estimate that $(2B/p^2 + O(1))^3$ points in $S_B$ reduce to a given point in $\PP^2(\ZZ/p^2)$.  Then reversing the order of summation and applying Lemma \ref{W2condition} gives 
\begin{eqnarray*} 
\frac{1}{\#S_B}\sum_{S_B} \widetilde{\pi}_{y}^{d_0}(x)& = &\frac{1}{\# S_B}\sum_{p\leq x} \pi_p^{d_0,\good}\\
& \leq &\frac{1}{\# S_B}\sum_{p\leq x} \left(\frac{2B}{p^2} + O(1)\right)^3\nu_{d_0}(p).
\end{eqnarray*} Expanding the right-hand-side, we have $$\frac{1}{\#S_B}\left(\sum_{p\leq x} \frac{8B^3\nu_{d_0}(p)}{p^6} + O\left(\sum_{p\leq x}\frac{4B^2\nu_{d_0}(p)}{p^4} + \sum_{[a:b:c]} \frac{2B\nu_{d_0}(p)}{p^2} + \sum_{p\leq x} \nu_{d_0}(p)\right)\right).$$ Therefore, combining Schanuel's result that $\#S_B = B^3/\zeta(3) + O(B^2)$ with the estimates of Lemma \ref{asymptoticbounds} and summing over $d_0\leq d$ yields the theorem.
\end{proof}

\end{document}